\documentclass[12pt]{article}
\usepackage{amsmath,amsthm,amssymb}
\usepackage{amssymb,latexsym}

\newtheorem{theorem}{Theorem}

\newtheorem{lemma}{Lemma}

\begin{document}

\baselineskip=17pt

\title{\bf A ternary diophantine inequality by primes near to squares}

\author{\bf S. I. Dimitrov}
\date{2019}
\maketitle
\begin{abstract}
Let $c$ be fixed with $1<c<35/34$.
In this paper we prove that for every sufficiently large real number $N$
and a small constant $\varepsilon>0$, the diophantine inequality
\begin{equation*}
|p_1^c+p_2^c+p_3^c-N|<\varepsilon
\end{equation*}
is solvable in primes $p_1,\,p_2,\,p_3$ near to squares.
\quad\\
{\bf  2010 Math.\ Subject Classification}:  11P55 $\cdot$ 11J25
\end{abstract}

\section{Introduction and statement of the result}
\indent

In 1952 I. I. Piatetski-Shapiro \cite{Shapiro} investigated the inequality
\begin{equation}\label{Shapiro}
|p_1^c+p_2^c+\cdot\cdot\cdot+p_r^c-N|<\varepsilon
\end{equation}
where $c>1$ is not an integer, $\varepsilon$ is a fixed small positive number, and
$p_1,...,p_r$ are primes. He proved the existence of an $H(c)$, depending only on
$c$, such that for all sufficiently large real $N$, (\ref{Shapiro}) has a solution for $H(c)\leq r$.
He established that
\begin{equation*}
\limsup\limits_{c\rightarrow\infty}\frac{H(c)}{c\log c}\leq4
\end{equation*}
and also that $H(c)\leq5$ if $1<c<3/2$.

In 1992 Tolev \cite{Tolev2} showed that (\ref{Shapiro}) has a solution for $r=3$ and $1<c<15/14$.
The interval $1<c<15/14$ was subsequently improved by several authors
\cite{Baker-Weingartner}, \cite{Cai1}, \cite{Cai2}, \cite{Cai3}, \cite{Cao-Zhai}, \cite{Ku-Ne}, \cite{Kumchev}.
The best result up to now belongs to Cai \cite{Cai3} with $1<c<43/36$.

On the other hand in 1991 Tolev \cite{Tolev1} solved the diophantine inequality
\begin{equation*}
 |\lambda_1p_1+\lambda_2p_2+\lambda_3p_3+\eta|<\varepsilon
\end{equation*}
in primes $p_1,\,p_2,\,p_3$ near to squares.
Here $\eta$ is real, the constants $\lambda_1,\lambda_2,\lambda_3$
satisfy some necessary conditions and $\varepsilon>0$  is a small constant.

More precisely Tolev proved the following theorem
\begin{theorem}
Suppose that $\lambda_1,\lambda_2,\lambda_3$ are non-zero real numbers, not all of
the same sign, that $\eta$ is real, $\lambda_1/\lambda_2$ is irrational and $0<\tau<1/8$.
Then there exist infinitely many triples of primes $p_1,\,p_2,\,p_3$ such that
\begin{equation*}
 |\lambda_1p_1+\lambda_2p_2+\lambda_3p_3+\eta|<(\max p_j)^{-\tau}
\end{equation*}
and
\begin{equation*}
\|\sqrt{p_1}\|,\; \|\sqrt{p_2}\|,\; \|\sqrt{p_3}\|<(\max p_j)^{-(1-8\tau)/26}\log^5(\max p_j)
\end{equation*}
(as usual, $\|\alpha\|$ denotes the distance from $\alpha$ to the nearest integer).
\end{theorem}
\begin{proof}
See \cite{Tolev1}.
\end{proof}

Motivated by these results and following the method of Tolev \cite{Tolev1}
we shall prove the following theorem
\begin{theorem}\label{MyTheorem} Let $c$ and $\tau$ be fixed with $1<c<\tau<35/34$
and $\delta>0$ be a fixed sufficiently small number.
Then for every sufficiently large real number $N$, the diophantine inequality
\begin{equation*}
|p_1^c+p_2^c+p_3^c-N|<N^{-\frac{1}{c}(\tau-c)}\log N
\end{equation*}
is solvable in primes $p_1,\,p_2,\,p_3$  such that
\begin{equation*}
\|\sqrt{p_1}\|,\; \|\sqrt{p_2}\|,\; \|\sqrt{p_3}\|<N^{-\frac{17}{48c}\big(\frac{35}{34}-\tau\big)+\delta}.
\end{equation*}
\end{theorem}

\section{Notations and lemmas}
\indent

Let $N$ be a sufficiently large positive number.
By $\eta$ we denote an arbitrary small positive number, not the same in all appearances.
For positive $A$ and $B$ we write $A\asymp B$ instead of $A\ll B\ll A$.
As usual $\mu(n)$ is M\"{o}bius' function and $\tau(n)$
denotes the number of positive divisors of $n$.
The letter $p$  with or without subscript will always denote prime number.
We denote by $\Lambda(n)$ von Mangoldt's function.
Moreover $e(y)=e^{2\pi \imath y}$. As usual, $[y]$ denotes the integer part of $y$.
Let $c$ and $\tau$ be fixed with $1<c<\tau<35/34$.
By $\delta$ we denote an fixed sufficiently small positive number.\\
Denote
\begin{align}
\label{X}
&X=(N/2)^{1/c}\,;\\
\label{varepsilon}
&\varepsilon=X^{c-\tau}\,;\\
\label{r}
&r=[\log X]\,;\\
\label{Y}
&Y=X^{-\frac{17}{48}\big(\frac{35}{34}-\tau\big)+\delta}\,;\\
\label{Delta}
&\Delta=Y/5\,;\\
\label{M}
&M=\Delta^{-1}r\,;\\
\label{Salpha}
&S(\alpha)=\sum\limits_{X/2<p\leq X} e(\alpha p^c)\log p\,;\\
\label{Ualpha}
&U(\alpha, m)=\sum\limits_{X/2<p\leq X} e(\alpha p^c+m\sqrt{p})\log p.
\end{align}

\begin{lemma}\label{Periodicfunction} Let $r\in \mathbb{N}$.
There exists a function $\chi(t)$ which is $r$-times continuously differentiable and
1-periodic with a Fourier series of the form
\begin{equation}\label{Fourierseries}
\chi(t)=\frac{9}{5}Y+\sum\limits_{m=-\infty\atop{m\neq0}}^\infty g(m) e(mt),
\end{equation}
where
\begin{equation}\label{gmest}
|g(m)|\leq\min\bigg(\frac{1}{\pi|m|},\frac{1}{\pi |m|}
\bigg(\frac{r}{\pi |m|\Delta}\bigg)^r\bigg)
\end{equation}
and
\begin{equation}\label{chit}
\chi(t) =
  \begin{cases}
    1  \quad   \text{ if }  &\|t\|\leq Y-\Delta,\\
    0 \quad  \text{ if }  &\|t\|\geq Y,\\
     \text{between} &0 \, \text{ and }\, 1 \text{ for the other t }.
   \end{cases}
\end{equation}
\end{lemma}
\begin{proof}
See (\cite{Karatsuba}, p. 14).
\end{proof}
We also denote
\begin{align}
\label{Halpha}
&H(\alpha)=\sum\limits_{X/2<p\leq X} \chi(\sqrt{p}) e(\alpha p^c)\log p\,;\\
\label{Valpha}
&V(\alpha)=\sum\limits_{m=-\infty\atop{m\neq0}}^\infty g(m) U(\alpha, m).
\end{align}
Further we need the function $A(x)$ used by Baker and Harman \cite{Baker-Harman}.
It is continuous and integrable on the real line such that
\begin{equation}\label{Achi}
A(x)\leq\chi_{[-1, 1]}(x).
\end{equation}
Further, if we write
\begin{equation*}
\hat{A}(\alpha)=\int\limits_{-\infty}^\infty A(x) e(-\alpha x) dx,
\end{equation*}
then
\begin{equation*}
\hat{A}(\alpha)=0 \quad \text{ for } \quad  |\alpha|\geq\mu,
\end{equation*}
where $\mu$ is a constant. Therefore if
\begin{equation}\label{P}
P=\frac{\mu}{\varepsilon},
\end{equation}
then
\begin{equation}\label{hatA0}
\hat{A}(\varepsilon\alpha)=0  \quad \text{ for } \quad |\alpha|\geq P.
\end{equation}

\begin{lemma}\label{Lowerbound}
Let $1<c<15/14$. Then
\begin{equation}\label{lowerbound}
\int\limits_{-\infty}^\infty S^3(\alpha) e(-N\alpha) \hat{A}(\varepsilon\alpha)\, d\alpha\gg X^{3-c}\,.
\end{equation}
\end{lemma}
\begin{proof}
Arguing as in \cite{Baker-Harman} and \cite{Tolev2} we obtain the lower bound \eqref{lowerbound}.
\end{proof}

\begin{lemma}\label{Korput} (Van der Corput) Let $k \geq 2$, $K = 2^{k-1}$
and $f(x)$ be a real-valued function with $k$ continuous derivatives in $[a, b]$ such that
\begin{equation*}
|f^{(k)}(x)|\asymp\lambda, \mbox{ uniformly in } x\in[a,b].
\end{equation*}
Then
\begin{equation*}
\bigg|\sum_{a<n\le b}e(f(n))\bigg|
\ll(b-a)\lambda^{\frac{1}{2K-2}}+(b-a)^{1-\frac{2}{K}}\lambda^{-\frac{1}{2K-2}}.
\end{equation*}
\end{lemma}
\begin{proof}
See (\cite{Karatsuba}, Ch. 1, Th. 5).
\end{proof}

\begin{lemma}\label{Iwaniec-Kowalski}
For any complex numbers $a(n)$ we have
\begin{equation*}
\bigg|\sum_{a<n\le b}a(n)\bigg|^2
\leq\bigg(1+\frac{b-a}{Q}\bigg)\sum_{|q|\leq Q}\bigg(1-\frac{|q|}{Q}\bigg)
\sum_{a<n,\, n+q\leq b}a(n+q)\overline{a(n)},
\end{equation*}
where $Q$ is any positive integer.
\end{lemma}
\begin{proof}
See (\cite{Iwaniec-Kowalski}, Lemma 8.17).
\end{proof}

\begin{lemma}\label{IntSalpha}
For the sum denoted by \eqref{Salpha} we have
\begin{equation*}
 \int\limits_{-P}^P|S(\alpha)|^2\,d\alpha\ll PX\log^3X.
\end{equation*}
\end{lemma}
\begin{proof}
See( \cite{Tolev2}, Lemma 7).
\end{proof}

\begin{lemma}\label{IntValpha}
For the sum denoted by \eqref{Valpha} we have
\begin{equation*}
 \int\limits_{-P}^P|V(\alpha)|^2\,d\alpha\ll PX\log^5X.
\end{equation*}
\end{lemma}
\begin{proof}
On the one hand
\begin{equation}\label{IntValphaest1}
\int\limits_{-P}^P|V(\alpha)|^2\,d\alpha\ll P\int\limits_0^1|V(\alpha)|^2\,d\alpha.
\end{equation}
On the other hand arguing as in (\cite{Tolev1}, Lemma 5), (\cite{Tolev2}, Lemma 7)
and using \eqref{r}, \eqref{Delta}, \eqref{M}, \eqref{gmest}  we obtain
\begin{align*}
&\int\limits_0^1|V(\alpha)|^2\,d\alpha=\\
&=\sum_{|m_1|,\,|m_2|>0}
g(m_1)\overline{g(m_2)}\\
&\times\sum\limits_{X/2<p_1, p_2\leq X}
e(m_1\sqrt{p_1}-m_2\sqrt{p_2})
\log p_1\log p_2\int\limits_0^1\alpha(p_1^c-p_2^c)\,d\alpha\\
&\ll\sum_{|m_1|,\,| m_2|>0}|g(m_1)|. |g(m_2)|
\sum\limits_{X/2<p_1, p_2\leq X}
\log p_1\log p_2\Bigg|\int\limits_0^1\alpha(p_1^c-p_2^c)\,d\alpha\Bigg|\\
&\ll X\log^3X
\sum_{|m_1|,\,|m_2|>0}|g(m_1)|. | g(m_2)|\\
&=X\log^3X\bigg( \sum_{|m|>0}|g(m)|^2
+\sum_{|m_1|,\,|m_2|<M}| g(m_1)|. |g(m_2)|\\
&\quad\quad\quad+\sum_{0<m_1\leq M,\,|m_2|>M}| g(m_1)| .|g(m_2)|
+\sum_{|m_1|,\,|m_2|>M}|g(m_1)| . |g(m_2)|\bigg)\\
\end{align*}

\begin{align}\label{IntValphaest2}
&\ll X\log^3X\bigg( \sum_{|m|>0}\frac{1}{m^2}
+\sum_{0<|m_1|,\,|m_2|<M}
\frac{1}{|m_1|.|m_2|}\nonumber\\
&\quad\quad\quad+\sum_{0<m_1\leq M,\,|m_2|>M}
\frac{1}{|m_1|}|g(m_2)|
+\sum_{|m_1|,\,|m_2|>M}|g(m_1)|.|g(m_2)|\bigg)\nonumber\\
&\ll X\log^3X\bigg( \log^2X+\bigg(\frac{r}{\pi M\Delta}\bigg)^r\log X
+\bigg(\frac{r}{\pi M\Delta}\bigg)^{2r}\bigg)\nonumber\\
&\ll X\log^3X\bigg( \log^2X+\frac{\log X}{X}
+\frac{1}{X^2}\bigg)\nonumber\\
&\ll X\log^5X.
\end{align}
From \eqref{IntValphaest1} and \eqref{IntValphaest2} it follows the assertion in the lemma.
\end{proof}

\begin{lemma}\label{Valphaest} For the sum denoted by \eqref{Valpha}  the upper bound
\begin{align}\label{Valphaestimation}
\max\limits_{|\alpha|\leq P} |V(\alpha)|\ll
&\Big( M^{1/2}X^{7/12} + M^{1/6} X^{3/4}+X^{11/12}+P^{1/16}X^{\frac{2c+29}{32}}\nonumber\\
&+P^{-3/16}M^{1/4}X^{\frac{33-6c}{32}}+P^{-1/16}M^{1/12}X^{\frac{31-2c}{32}}\Big)X^\eta
\end{align}
holds.
\end{lemma}
\begin{proof}
Bearing in mind \eqref{r}, \eqref{Delta}, \eqref{M}, \eqref{Ualpha}, \eqref{gmest}
and \eqref{Valpha} we write
\begin{align}\label{Valphaest1}
|V(\alpha)|&\ll\sum_{0<|m|\leq M}\frac{1}{|m|}|U(\alpha, m)|
+X\sum_{|m|>M}|g(m)|\nonumber\\
&\ll\sum_{0<|m|\leq M}\frac{1}{|m|}|U(\alpha, m)| +\bigg(\frac{r}{\pi M\Delta}\bigg)^rX\nonumber\\
&\ll \sum_{0<|m|\leq M}\frac{1}{|m|}|U(\alpha, m)|+1.
\end{align}
In order to prove the lemma we have to find the upper bound of the sum $U(\alpha, m)$ denoted by \eqref{Ualpha}.
Our argument is a modification of Petrov's and Tolev's  \cite{Petrov-Tolev} argument.

Assume that $m>0$. For $m<0$ the proof is analogous.

We denote
\begin{equation}\label{psit}
\psi(t)=\alpha t^c+m\sqrt{t}.
\end{equation}
\begin{equation}\label{fdl}
f(d,l)=\psi(dl)=\alpha (dl)^c+m\sqrt{dl}.
\end{equation}
It is clear that
\begin{equation*}
U(\alpha, m)=\sum\limits_{X/2<n\leq X}\Lambda(n)e(\alpha n^c+m\sqrt{n})
+\mathcal{O}( X^{1/2}).
\end{equation*}
Using Vaughan's identity (see \cite{Vaughan}) we get
\begin{equation}\label{Ualphadecomp}
U(\alpha, m)=U_1-U_2-U_3-U_4+\mathcal{O}( X^{1/2}),
\end{equation}
where
\begin{align}
\label{U1}
&U_1=\sum_{d\le X^{1/3}}\mu(d)\sum_{X/2d<l\le X/d}(\log l)e(f(d,l)),\\
\label{U2}
&U_2=\sum_{d\le X^{1/3}}c(d)\sum_{X/2d<l\le X/d}e(f(d,l)),\\
\label{U3}
&U_3=\sum_{X^{1/3}<d\le X^{2/3}}c(d)\sum_{X/2d<l\le X/d}e(f(d,l)),\\
\label{U4}
&U_4= \mathop{\sum\sum}_{\substack{X/2<dl\le X \\d>X^{1/3},\,l>X^{1/3} }}
a(d)\Lambda(l) e(f(d,l)),
\end{align}
and where
\begin{equation}\label{cdad}
|c(d)|\leq\log d,\quad  | a(d)|\leq\tau(d).
\end{equation}

\textbf{Estimation of $U_1$ and $U_2$}

Consider first $U_2$ defined by \eqref{U2}.  Bearing in mind  \eqref{fdl} we find
\begin{equation}\label{f''}
f^{\prime\prime}_{ll}(d,l)=\gamma_1-\gamma_2,
\end{equation}
where
\begin{equation}\label{gamma12}
\gamma_1=d^2\alpha c(c-1)(dl)^{c-2}, \quad \gamma_2=\frac{1}{4}md^2(dl)^{-3/2}.
\end{equation}
From \eqref{gamma12} and the restriction
\begin{equation}\label{restriction}
X/2<dl\le X
\end{equation}
we obtain
\begin{equation}\label{gamma12est}
|\gamma_1|\asymp |\alpha| d^2X^{c-2},
\quad |\gamma_2|\asymp md^2X^{-3/2}.
\end{equation}
On the one hand from  \eqref{f''} and \eqref{gamma12est} we conclude
that there exists sufficiently small constant $h_0>0$ such that if
$|\alpha|\leq h_0mX^{1/2-c}$, then $|f^{\prime\prime}_{ll}(d,l)|\asymp md^2X^{-3/2}$.

On the other hand from  \eqref{f''} and \eqref{gamma12est} it follows that
there exists sufficiently large constant $H_0>0$ such that if
$|\alpha|\geq H_0mX^{1/2-c}$, then $|f^{\prime\prime}_{ll}(d,l)|\asymp |\alpha| d^2X^{c-2}$.

Consider several cases.

\textbf{Case 1a.}
\begin{equation}\label{Case1a}
H_0mX^{1/2-c}\leq|\alpha|\leq P.
\end{equation}
We remind that in this case $|f^{\prime\prime}_{ll}(d,l)|\asymp |\alpha| d^2X^{c-2}$
and using Lemma \ref{Korput} for $k=2$ we get
\begin{align}\label{Case1aest}
\sum\limits_{X/2d<l\le X/d}e(f(d, l))
 &\ll\frac{X}{d}\big(|\alpha|d^2X^{c-2}\big)^{1/2}+\big(|\alpha|d^2X^{c-2}\big)^{-1/2}\nonumber\\
&= |\alpha|^{1/2} X^{c/2}+|\alpha|^{-1/2} d^{-1}X^{1-c/2}.
\end{align}
From \eqref{U2}, \eqref{cdad}, \eqref{Case1a} and \eqref{Case1aest} it follows
\begin{equation}\label{U2est1}
U_2\ll \big(P^{1/2} X^{\frac{3c+2}{6}}+m^{-1/2}X^{3/4}\big)\log^2X.
\end{equation}

\textbf{Case 2a.}
\begin{equation}\label{Case2a}
h_0mX^{1/2-c}<\alpha< H_0mX^{1/2-c}.
\end{equation}
By \eqref{fdl} we find
\begin{equation}\label{f'''}
f^{\prime\prime\prime}_{lll}(d,l)=d^3\alpha c(c-1)(c-2)(dl)^{c-3}
+\frac{3}{8}d^3m(dl)^{-5/2}.
\end{equation}
The formulas \eqref{f''},  \eqref{gamma12} and  \eqref{f'''} give us
\begin{equation}\label{f''f'''}
(c-2)f^{\prime\prime}_{ll}(d,l)-lf^{\prime\prime\prime}_{lll}(d,l)
=\frac{1-2c}{8}d^2(dl)^{-3/2}m.
\end{equation}
From \eqref{restriction} and \eqref{f''f'''} we obtain
\begin{equation*}
|(c-2)f^{\prime\prime}_{ll}(d,l)-lf^{\prime\prime\prime}_{lll}(d,l)|
\asymp md^2X^{-3/2}.
\end{equation*}
The above implies that there exists $\alpha_0 > 0$, such that for every
$l\in(X/2d, X/d]$ at least one of the following inequalities is fulfilled:
\begin{equation}\label{f''est1}
|f^{\prime\prime}_{ll}(d,l)|\geq \alpha_0 md^2X^{-3/2}.
\end{equation}
\begin{equation}\label{f'''est1}
|f^{\prime\prime\prime}_{lll}(d,l)|\geq \alpha_0 md^3X^{-5/2}.
\end{equation}
Let us consider the equation
\begin{equation}\label{f'''equation1}
f^{\prime\prime\prime}_{lll}(d,l)=0.
\end{equation}
From \eqref{f'''} it is tantamount to
\begin{equation}\label{f'''equation2}
3m(dl)^{1/2-c}-8\alpha c(c-1)(c-2)=0.
\end{equation}
It is easy to see that the equation \eqref{f'''equation2}
has at most 1 solution $Z\in(X^{1/2-c}, (X/2)^{1/2-c}]$.
Consequently the equation \eqref{f'''equation1}
has at most 1 solution in real numbers $l\in(X/2d, X/d]$.
According to Rolle's Theorem if $C$ does not depend on $l$  then the equation $f^{\prime\prime}_{ll}(d,l)=C$ has at most 2 solution in real numbers $l\in(X/2d, X/d]$.
Therefore the equation  $|f^{\prime\prime}_{ll}(d,l)|=\alpha_0 md^2X^{-3/2}$
has at most 4 solution in real numbers $l\in(X/2d, X/d]$.
From these consideration it follows that the interval $(X/2d, X/d]$ can be divided into at most 5 intervals such that if $J$ is one of them, then at least one of the following assertions holds:
\begin{equation}\label{f''estJ}
\mbox{The inequality } \eqref{f''est1}  \mbox{ is fulfilled for all } l\in J.
\end{equation}
\begin{equation}\label{f'''estJ}
\mbox{The inequality } \eqref{f'''est1}  \mbox{ is fulfilled for all } l\in J.
\end{equation}
On the other hand from \eqref{f''}, \eqref{restriction}, \eqref{gamma12est}, \eqref{Case2a} and \eqref{f'''} we get
\begin{equation}\label{f''f'''upper}
|f^{\prime\prime}_{ll}(d,l)|\ll md^2X^{-3/2}, \quad
|f^{\prime\prime\prime}_{lll}(d,l)|\ll md^3X^{-5/2}.
\end{equation}
Bearing in mind \eqref{f''estJ} -- \eqref{f''f'''upper} we conclude that
the interval $(X/2d, X/d]$ can be divided into at most 5 intervals such that if
$J$ is one of them, then at least one of the following statements  is fulfilled:
\begin{equation}\label{f''est2}
|f^{\prime\prime}_{ll}(d,l)|\asymp md^2X^{-3/2} \quad \mbox{ uniformly for } \quad l\in J.
\end{equation}
\begin{equation}\label{f'''est2}
|f^{\prime\prime\prime}_{lll}(d,l)|\asymp md^3X^{-5/2} \quad  \mbox{ uniformly for } \quad l\in J.
\end{equation}
If \eqref{f''est2} holds, then we use Lemma \ref{Korput} for $k=2$ and obtain
\begin{align}\label{Case2aest1}
\sum\limits_{l\in J}e(f(d, l))
 &\ll\frac{X}{d}\big(md^2X^{-3/2}\big)^{1/2}+\big(md^2X^{-3/2}\big)^{-1/2}\nonumber\\
&\ll m^{1/2} X^{1/4}+m^{-1/2} d^{-1}X^{3/4}.
\end{align}
If \eqref{f'''est2}  is fulfilled, then we use Lemma \ref{Korput} for $k=3$ and find
\begin{align}\label{Case2aest2}
\sum\limits_{l\in J}e(f(d, l))
 &\ll\frac{X}{d}\big(md^3X^{-5/2}\big)^{1/6}+\bigg(\frac{X}{d}\bigg)^{1/2}
 \big(md^3X^{-5/2} \big)^{-1/6}\nonumber\\
&=m^{1/6} d^{-1/2}X^{7/12}+m^{-1/6} d^{-1}X^{11/12}.
\end{align}
From \eqref{Case2aest1} and \eqref{Case2aest2} it follows
\begin{align}\label{Case2aest3}
\sum\limits_{X/2d<l\le X/d}e(f(d, l))&\ll m^{1/2} X^{1/4}+m^{-1/2} d^{-1}X^{3/4}\nonumber\\
&+m^{1/6} d^{-1/2}X^{7/12}+m^{-1/6} d^{-1}X^{11/12}.
\end{align}
Bearing in mind \eqref{U2} and \eqref{Case2aest3} we get
\begin{equation}\label{U2est2}
U_2\ll \big(m^{1/2}X^{7/12}+m^{1/6} X^{3/4}+m^{-1/6}X^{11/12}\big)\log^2X.
\end{equation}

\textbf{Case 3a.}
\begin{equation}\label{Case3a}
|\alpha|\leq h_0mX^{1/2-c}.
\end{equation}
We recall that  in this case $|f^{\prime\prime}_{ll}(d,l)|\asymp md^2X^{-3/2}$
and using Lemma \ref{Korput} for $k=2$ we obtain
\begin{equation}\label{Case3aest}
\sum\limits_{X/2d<l\le X/d}e(f(d, l))\ll m^{1/2} X^{1/4}+m^{-1/2} d^{-1}X^{3/4}.
\end{equation}
Using \eqref{U2} and \eqref{Case3aest} we find
\begin{equation}\label{U2est3}
U_2\ll \big(m^{1/2}X^{7/12}+m^{-1/2}X^{3/4}\big)\log^2X.
\end{equation}

\textbf{Case 4a.}
\begin{equation}\label{Case4a}
-H_0mX^{1/2-c}<\alpha< -h_0mX^{1/2-c}.
\end{equation}
In this case again $|f^{\prime\prime}_{ll}(d,l)|\asymp md^2X^{-3/2}$.
Consequently
\begin{equation}\label{U2est4}
U_2\ll \big(m^{1/2}X^{7/12}+m^{-1/2}X^{3/4}\big)\log^2X.
\end{equation}
From \eqref{U2est1}, \eqref{U2est2}, \eqref{U2est3} and \eqref{U2est4}
it follows
\begin{equation}\label{U2estimation}
U_2\ll \big(m^{1/2}X^{7/12}+m^{1/6} X^{3/4}+m^{-1/6}X^{11/12}+P^{1/2} X^{\frac{3c+2}{6}}\big)\log^2X.
\end{equation}
In order to estimate $U_1$ defined by \eqref{U1}
we apply Abel's transformation.
Then arguing as in the estimation of $U_2$  we get
\begin{equation}\label{U1estimation}
U_1\ll \big(m^{1/2}X^{7/12}+m^{1/6} X^{3/4}+m^{-1/6}X^{11/12}+P^{1/2} X^{\frac{3c+2}{6}}\big)\log^2X.
\end{equation}

\textbf{Estimation of $U_3$ and $U_4$}

Consider first $U_4$ defined by \eqref{U4}. We have
\begin{equation}\label{U4U5}
U_4\ll|U_5|\log X,
\end{equation}
where
\begin{equation}\label{U5}
U_5=\sum_{L<d\le 2L}b(l)\sum_{D<d\le 2D\atop{X/2l<d\leq X/l}}a(d)e(f(d,l))
\end{equation}
and where
\begin{equation}\label{ParU5}
a(d)\ll X^\eta, \quad b(l)\ll X^\eta, \quad
X^{1/3}\ll D\ll X^{1/2}\ll L\ll X^{2/3}, \quad  DL\asymp X.
\end{equation}
Using \eqref{U5}, \eqref{ParU5} and Cauchy's inequality we obtain
\begin{equation}\label{U52est1}
|U_5|^2\ll X^\eta L  \sum_{L<d\le 2L}\bigg|\sum_{D_1<d\le D_2}a(d)e(f(d,l))\bigg|^2,
\end{equation}
where
\begin{equation}\label{maxmin1}
D_1=\max{\bigg\{D,\frac{X}{2l}\bigg\}},\quad
D_2=\min{\bigg\{\frac{X}{l},2D\bigg\}}\,.
\end{equation}
Now from \eqref{ParU5} -- \eqref{maxmin1}  and Lemma \ref{Iwaniec-Kowalski} with $Q$ such that
\begin{equation}\label{QD}
Q\leq D
\end{equation}
we find
\begin{align}\label{U52est2}
|U_5|^2&\ll X^\eta L  \sum_{L<d\le 2L}\frac{D}{Q}
\sum_{|q|\leq Q}\bigg(1-\frac{|q|}{Q}\bigg)
\sum_{D_1<d\le D_2\atop{D_1<d+q\le D_2}}a(d+q)\overline{a(d)}e(f(d+q,l)-f(d,l))\nonumber\\
&\ll \Bigg(\frac{(LD)^2}{Q}+\frac{LD}{Q}\sum_{0<|q|\leq Q}
\sum_{D<d\le 2D\atop{D<d+q\le 2D}}\bigg|\sum_{L_1<l\le L_2}e\big(g_{d,q}(l)\big)\bigg|\Bigg)X^\eta,
\end{align}
where
\begin{equation}\label{maxmin2}
L_1=\max{\bigg\{L,\frac{X}{2d},\frac{X}{2(d+q)}\bigg\}},\quad
L_2=\min{\bigg\{2L,\frac{X}{d},\frac{X}{d+q}\bigg\}}
\end{equation}
and
\begin{equation}\label{gl}
g(l)=g_{d,q}(l)=f(d+q,l)-f(d,l).
\end{equation}
It is not hard to see that the sum over negative $q$ in formula \eqref{U52est2}
is equal to the sum over positive $q$. Thus
\begin{equation}\label{U52est3}
|U_5|^2\ll\Bigg(\frac{ (LD)^2}{Q}+\frac{LD}{Q}\sum_{1\leq q\leq Q}
\sum_{D<d\le 2D-q}\bigg|\sum_{L_1<l\le L_2}e(g_{d,q}(l))\bigg|\Bigg)X^\eta.
\end{equation}
Consider the function $g(l)$.
From \eqref{psit}, \eqref{fdl} and \eqref{gl} it follows
\begin{equation*}
g(l)=\int\limits_{d}^{d+q}f_t^\prime(t,l)\,dt=\int\limits_{d}^{d+q}l\psi^\prime(tl)\,dt.
\end{equation*}
Hence
\begin{equation}\label{g''lint1}
g^{\prime\prime}(l)=\int\limits_{d}^{d+q}2t\psi^{\prime\prime}(tl)
+lt^2\psi^{\prime\prime\prime}(tl)\,dt.
\end{equation}
Bearing in mind \eqref{psit} and \eqref{g''lint1} we obtain
\begin{equation}\label{g''lint2}
g^{\prime\prime}(l)=\int\limits_{d}^{d+q} \Big(\Psi_1(t, l)-\Psi_2(t, l)\Big)\,dt,
\end{equation}
where
\begin{equation}\label{Psi1Psi2}
\Psi_1(t, l)=\alpha c^2(c-1)t^{c-1}l^{c-2}, \quad
\Psi_2(t, l)=\frac{m}{8}t^{-1/2}l^{-3/2}.
\end{equation}
If $t\in[d, d+q]$, then
\begin{equation}\label{tl}
tl\asymp X.
\end{equation}
From \eqref{Psi1Psi2} and \eqref{tl} we get
\begin{equation}\label{Psi1Psi2est}
|\Psi_1(t, l)|\asymp |\alpha| d^2X^{c-2},
\quad |\Psi_2(t, l)|\asymp md^2X^{-3/2}.
\end{equation}
On the one hand from  \eqref{g''lint2} and \eqref{Psi1Psi2est} we conclude
that there exists sufficiently small constant $h_1>0$ such that if
$|\alpha|\leq h_1mX^{1/2-c}$, then $|g^{\prime\prime}(l)|\asymp qmdX^{-3/2}$.

On the other hand from  \eqref{g''lint2} and \eqref{Psi1Psi2est} it follows that
there exists sufficiently large constant $H_1>0$ such that if
$|\alpha|\geq H_1mX^{1/2-c}$, then $|g^{\prime\prime}(l)|\asymp q|\alpha| dX^{c-2}$.

Consider several cases.

\textbf{Case 1b.}
\begin{equation}\label{Case1b}
H_1mX^{1/2-c}\leq|\alpha|\leq P.
\end{equation}
We recall that the constant $H_1$ is chosen in such a way, that if $|\alpha|\geq H_1mX^{1/2-c}$,
then uniformly for $l\in(L_1, L_2]$ we have $|g^{\prime\prime}(l)|\asymp q|\alpha| dX^{c-2}$.
Using \eqref{ParU5}, \eqref{maxmin2} and applying Lemma \ref{Korput} for $k=2$ we find
\begin{align}\label{Case1best}
\sum_{L_1<l\le L_2}e(g(l))
 &\ll L\big(q|\alpha|dX^{c-2}\big)^{1/2}+\big(q|\alpha|dX^{c-2}\big)^{-1/2}\nonumber\\
&= Lq^{1/2}|\alpha|^{1/2}d^{1/2} X^{c/2-1}+q^{-1/2}|\alpha|^{-1/2} d^{-1/2}X^{1-c/2}.
\end{align}
From \eqref{ParU5}, \eqref{U52est3}, \eqref{Case1b} and \eqref{Case1best} it follows
\begin{equation}\label{U5est1}
U_5\ll \big(XQ^{-1/2}+P^{1/4}X^{\frac{2c+5}{8}}Q^{1/4}+m^{-1/4}XQ^{-1/4}\big)X^\eta.
\end{equation}

\textbf{Case 2b.}
\begin{equation}\label{Case2b}
h_1mX^{1/2-c}<\alpha< H_1mX^{1/2-c}.
\end{equation}
The formulas \eqref{g''lint2} and \eqref{Psi1Psi2} give us
\begin{equation}\label{g'''lint}
g^{\prime\prime\prime}(l)=\int\limits_{d}^{d+q} \Big(\Phi_1(t, l)+\Phi_2(t, l)\Big)\,dt,
\end{equation}
where
\begin{equation}\label{Phi1Phi2}
\Phi_1(t, l)=\alpha c^2(c-1)(c-2)t^{c-1}l^{c-3}, \quad
\Phi_2(t, l)=\frac{3m}{16}t^{-1/2}l^{-5/2}.
\end{equation}
From \eqref{g''lint2}, \eqref{Psi1Psi2}, \eqref{g'''lint}  and  \eqref{Phi1Phi2}  it follows
\begin{equation}\label{g''g'''}
(c-2)g^{\prime\prime}(l)-lg^{\prime\prime\prime}(l)
=\frac{7-2c}{16}m\int\limits_{d}^{d+q}t(tl)^{-3/2}\,dt.
\end{equation}
Using \eqref{tl} and \eqref{g''g'''} we obtain
\begin{equation*}
|(c-2)g^{\prime\prime}(l)-lg^{\prime\prime\prime}(l)|\asymp qmdX^{-3/2}.
\end{equation*}
Consequently there exists $\alpha_1 > 0$, such that for every
$l\in(L_1, L_2]$ at least one of the following inequalities holds:
\begin{equation}\label{g''est1}
|g^{\prime\prime}(l)|\geq \alpha_1 qmdX^{-3/2}.
\end{equation}
\begin{equation}\label{g'''est1}
|g^{\prime\prime\prime}(l)|\geq \alpha_1 qmd^2X^{-5/2}.
\end{equation}
Consider the equation
\begin{equation}\label{g'''equation1}
g^{\prime\prime\prime}(l)=0.
\end{equation}
From \eqref{g'''lint}  and  \eqref{Phi1Phi2} we get
\begin{equation}\label{g'''equation2}
\alpha c(c-1)(c-2)[(d+q)^c-d^c]l^{c-3}-\frac{3m}{8}[(d+q)^{1/2}-d^{1/2}]l^{-5/2}=0
\end{equation}
which is equivalent to
\begin{equation}\label{g'''equation3}
l^{c-1/2}=\frac{3m[(d+q)^{1/2}-d^{1/2}]}{8\alpha c(c-1)(c-2)[(d+q)^c-d^c]}.
\end{equation}
It is not hard to see that the equation \eqref{g'''equation3}
has at most 1 solution $Z\in(L^{c-1/2}_1, L^{c-1/2}_2]$.
Therefore the equation \eqref{g'''equation1}
has at most 1 solution in real numbers $l\in(L_1, L_2]$.
According to Rolle's Theorem if $C$ does not depend on $l$  then the equation
$g^{\prime\prime}(l)=C$ has at most 2 solution in real numbers $l\in(L_1, L_2]$.
Therefore the equation  $|g^{\prime\prime}(l)|=\alpha_1q md^2X^{-3/2}$
has at most 4 solution in real numbers $l\in(L_1, L_2]$.
From these consideration it follows that the interval $(L_1, L_2]$
can be divided into at most 5 intervals such that if $J$ is one of them,
then at least one of the following statements holds:
\begin{equation}\label{g''estJ}
\mbox{The inequality } \eqref{g''est1}  \mbox{ is fulfilled for all } l\in J.
\end{equation}
\begin{equation}\label{g'''estJ}
\mbox{The inequality } \eqref{g'''est1}  \mbox{ is fulfilled for all } l\in J.
\end{equation}
Using \eqref{g''lint2}, \eqref{tl}, \eqref{Psi1Psi2est}, \eqref{Case2b},
\eqref{g'''lint}  and  \eqref{Phi1Phi2} we find
\begin{equation}\label{g''g'''upper}
|g^{\prime\prime}(l)|\ll qmdX^{-3/2}, \quad
|g^{\prime\prime\prime}(l)|\ll qmd^2X^{-5/2}.
\end{equation}
From \eqref{g''estJ} -- \eqref{g''g'''upper}  it follows that
the interval $(L_1, L_2]$ can be divided into at most 5 intervals such that if
$J$ is one of them, then at least one of the following assertions is fulfilled:
\begin{equation}\label{g''est2}
|g^{\prime\prime}(l)|\asymp qmdX^{-3/2} \quad \mbox{ uniformly for } \quad l\in J.
\end{equation}
\begin{equation}\label{g'''est2}
|g^{\prime\prime\prime}(l)|\asymp qmd^2X^{-5/2} \quad  \mbox{ uniformly for } \quad l\in J.
\end{equation}
If \eqref{g''est2} is fulfilled, then we use Lemma \ref{Korput} for $k=2$ and get
\begin{align}\label{Case2best1}
\sum\limits_{l\in J}e(g(l))
 &\ll L\big(qmdX^{-3/2}\big)^{1/2}+\big(qmdX^{-3/2}\big)^{-1/2}\nonumber\\
&=Lq^{1/2} m^{1/2} d^{1/2} X^{-3/4}+q^{-1/2} m^{-1/2} d^{-1/2}X^{3/4}.
\end{align}
If \eqref{g'''est2}  holds, then we use Lemma \ref{Korput} for $k=3$ and obtain
\begin{align}\label{Case2best2}
\sum\limits_{l\in J}e(g(l))
 &\ll L\big(qmd^2X^{-5/2}\big)^{1/6}+L^{1/2} \big(qmd^2X^{-5/2} \big)^{-1/6}\nonumber\\
&=Lq^{1/6} m^{1/6} d^{1/3} X^{-5/12}+L^{1/2} q^{-1/6} m^{-1/6} d^{-1/3}X^{5/12}.
\end{align}
From \eqref{Case2best1} and \eqref{Case2best2} it follows
\begin{align}\label{Case2best3}
\sum\limits_{L_1<l\le L_2}e(g(l))
&\ll Lq^{1/2} m^{1/2} d^{1/2} X^{-3/4}+q^{-1/2} m^{-1/2} d^{-1/2}X^{3/4}\nonumber\\
&+Lq^{1/6} m^{1/6} d^{1/3} X^{-5/12}+L^{1/2} q^{-1/6} m^{-1/6} d^{-1/3}X^{5/12}.
\end{align}
Taking into account  \eqref{ParU5}, \eqref{U52est3} and \eqref{Case2best3} we find
\begin{align}\label{U5est2}
U_5\ll&\big( XQ^{-1/2} +m^{1/4}X^{3/4}Q^{1/4}+m^{-1/4}XQ^{-1/4}\nonumber\\
&+m^{1/12}X^{7/8}Q^{1/12}+m^{-1/12}XQ^{-1/12}\big)X^\eta.
\end{align}

\textbf{Case 3b.}
\begin{equation}\label{Case3b}
|\alpha|\leq h_1mX^{1/2-c}.
\end{equation}
We have chosen the constant $h_1$  in such a way, that from
\eqref{g''lint2}, \eqref{tl}, \eqref{Psi1Psi2est} and \eqref{Case3b}
it follows that $|g^{\prime\prime}(l)|\asymp qmdX^{-3/2}$ uniformly for $l\in(L_1, L_2]$.
Applying Lemma \ref{Korput} for $k=2$ we get
\begin{equation}\label{Case3best}
\sum\limits_{L_1<l\le L_2}e(g(l))
\ll Lq^{1/2} m^{1/2} d^{1/2} X^{-3/4}+q^{-1/2} m^{-1/2} d^{-1/2}X^{3/4}.
\end{equation}
From \eqref{U52est3} and \eqref{Case3best} we obtain
\begin{equation}\label{U5est3}
U_5\ll \big( XQ^{-1/2} +m^{1/4}X^{3/4}Q^{1/4}+m^{-1/4}XQ^{-1/4}\big)X^\eta.
\end{equation}

\textbf{Case 4b.}
\begin{equation}\label{Case4b}
-H_1mX^{1/2-c}<\alpha< -h_1mX^{1/2-c}.
\end{equation}
In this case $|g^{\prime\prime}(l)|\asymp qmdX^{-3/2}$.
Arguing in a similar way we find
\begin{equation}\label{U5est4}
U_5\ll \big( XQ^{-1/2} +m^{1/4}X^{3/4}Q^{1/4}+m^{-1/4}XQ^{-1/4}\big)X^\eta.
\end{equation}
From \eqref{U4U5}, \eqref{U5est1}, \eqref{U5est2}, \eqref{U5est3} and \eqref{U5est4}
we get
\begin{align}\label{U4estimation}
U_4\ll&\Big( XQ^{-1/2} +P^{1/4}X^{\frac{2c+5}{8}}Q^{1/4}+m^{1/4}X^{3/4}Q^{1/4}
+m^{-1/4}XQ^{-1/4}\nonumber\\
&+m^{1/12}X^{7/8}Q^{1/12}+m^{-1/12}XQ^{-1/12}\Big)X^\eta.
\end{align}
Arguing as in the estimation of $U_4$  we obtain
\begin{align}\label{U3estimation}
U_3\ll&\Big( XQ^{-1/2} +P^{1/4}X^{\frac{2c+5}{8}}Q^{1/4}+m^{1/4}X^{3/4}Q^{1/4}
+m^{-1/4}XQ^{-1/4}\nonumber\\
&+m^{1/12}X^{7/8}Q^{1/12}+m^{-1/12}XQ^{-1/12}\Big)X^\eta.
\end{align}
Summarizing \eqref{Ualphadecomp}, \eqref{U2estimation}, \eqref{U1estimation},
\eqref{U4estimation} and \eqref{U3estimation} we conclude that
for $|\alpha|\leq P$ and any integer $m\neq0$ the estimation
\begin{align}\label{Ualphaestimation}
|U(\alpha, m)|\ll&\Big( m^{1/2}X^{7/12} + m^{1/6}X^{3/4} + m^{-1/6}X^{11/12}+XQ^{-1/2} \nonumber\\
&+P^{1/4}X^{\frac{2c+5}{8}}Q^{1/4}+m^{1/4}X^{3/4}Q^{1/4}+m^{-1/4}XQ^{-1/4}\nonumber\\
&+m^{1/12}X^{7/8}Q^{1/12}+m^{-1/12}XQ^{-1/12}\Big)X^\eta
\end{align}
holds.

We substitute the expression \eqref{Ualphaestimation} for $U(\alpha, m)$ in
\eqref{Valphaest1} and find
\begin{align}\label{Valphaest2}
\max\limits_{|\alpha|\leq P}|V(\alpha)|\ll&\Big( M^{1/2}X^{7/12} + M^{1/6}X^{3/4} + X^{11/12}+XQ^{-1/2} \nonumber\\
&+P^{1/4}X^{\frac{2c+5}{8}}Q^{1/4}+M^{1/4}X^{3/4}Q^{1/4}+XQ^{-1/4}\nonumber\\
&+M^{1/12}X^{7/8}Q^{1/12}+XQ^{-1/12}\Big)X^\eta.
\end{align}
We choose
\begin{equation}\label{Qchoose}
Q=\big[P^{-3/4}X^{\frac{9-6c}{8}}\big].
\end{equation}
The direct verification assures us that the condition \eqref{QD} is fulfilled.

Bearing in mind \eqref{Valphaest2} and  \eqref{Qchoose}
we obtain the estimation \eqref{Valphaestimation}.
\end{proof}

\section{Proof of the Theorem}
\indent

Consider the sum
\begin{equation}\label{Gamma}
\Gamma(X)= \sum\limits_{X/2<p_1,p_2,p_3\leq X\atop{|p_1^c+p_2^c+p_3^c-N|<\varepsilon
\atop{\|\sqrt{p_i}\|<Y,\,i=1,2,3}}}\log p_1\log p_2\log p_3\,.
\end{equation}
The theorem will be proved if we show that $\Gamma(X)\rightarrow\infty$ as $X\rightarrow\infty$.

Consider the integrals
\begin{align}
\label{I1def}
&I_1=\int\limits_{-\infty}^\infty H^3(\alpha) e(-N\alpha) \hat{A}(\varepsilon\alpha)\, d\alpha\\
\label{Idef}
&I=\int\limits_{-\infty}^\infty S^3(\alpha) e(-N\alpha) \hat{A}(\varepsilon\alpha)\, d\alpha.
\end{align}
On the one hand from \eqref{chit}, \eqref{Halpha}, \eqref{Achi},
\eqref{Gamma} and \eqref{I1def}  it follows
\begin{align}\label{I1est1}
I_1&=  \sum\limits_{X/2<p_1,p_2,p_3\leq X}  \prod_{j=1}^{3} \chi(\sqrt{p_j})\log p_j
\int\limits_{-\infty}^\infty e((p_1^c+p_2^c+p_3^c-N)\alpha)\hat{A}(\varepsilon\alpha)\,d\alpha\nonumber\\
&=  \sum\limits_{X/2<p_1,p_2,p_3\leq X}  \prod_{j=1}^{3} \chi(\sqrt{p_j})(\log p_j)\varepsilon^{-1}
A((p_1^c+p_2^c+p_3^c-N)\varepsilon^{-1})\leq\varepsilon^{-1}\Gamma(X).
\end{align}
On the other hand \eqref{Salpha}, \eqref{Fourierseries},
\eqref{Halpha}, \eqref{Valpha}, \eqref{hatA0}, \eqref{I1def} and \eqref{Idef} give us
\begin{align}\label{I1est2}
I_1&= \int\limits_{-\infty}^\infty \bigg(\frac{9}{5}YS(\alpha) +V(\alpha) \bigg)^3
e(-N\alpha)\hat{A}(\varepsilon\alpha)\,d\alpha\nonumber\\
&=  \bigg(\frac{9}{5}Y\bigg)^3I
+\mathcal{O}\Bigg(Y^2\int\limits_{-P}^P |S^2(\alpha) V(\alpha)| \,d\alpha  \Bigg)\nonumber\\
&+\mathcal{O}\Bigg(Y\int\limits_{-P}^P |S(\alpha) V^2(\alpha)| \,d\alpha  \Bigg)
+\mathcal{O}\Bigg(\int\limits_{-P}^P |V^3(\alpha)| \,d\alpha  \Bigg).
\end{align}
We write
\begin{equation}\label{Firstint}
\int\limits_{-P}^P |S^2(\alpha) V(\alpha)| \,d\alpha
\ll\max\limits_{|\alpha|\leq P} |V(\alpha)|\int\limits_{-P}^P |S(\alpha)|^2\,d\alpha.
\end{equation}
Applaying Cauchy's inequality we get
\begin{equation}\label{Secondint}
\int\limits_{-P}^P |S(\alpha) V^2(\alpha)| \,d\alpha
\ll\max\limits_{|\alpha|\leq P} |V(\alpha)|
\left(\int\limits_{-P}^P |S(\alpha)|^2\,d\alpha\right)^{1/2}
\left(\int\limits_{-P}^P |V(\alpha)|^2\,d\alpha\right)^{1/2}.
\end{equation}
Similarly
\begin{equation}\label{Thirdint}
\int\limits_{-P}^P |V(\alpha)|^3\,d\alpha
\ll\max\limits_{|\alpha|\leq P} |V(\alpha)|\int\limits_{-P}^P |V(\alpha)|^2\,d\alpha.
\end{equation}
Using Lemmas \ref{IntSalpha}, \ref{IntValpha},  \ref{Valphaest}
and \eqref{I1est2} -- \eqref{Thirdint}  we obtain
\begin{align}\label{I1est3}
I_1=\bigg(\frac{9}{5}Y\bigg)^3I  &+\mathcal{O} \Big( \big(PM^{1/2}X^{19/12}
+ PM^{1/6} X^{7/4}+PX^{23/12}+P^{17/16}X^{\frac{2c+61}{32}}\nonumber\\
&+P^{13/16}M^{1/4}X^{\frac{65-6c}{32}}+P^{15/16}M^{1/12}X^{\frac{63-2c}{32}}\big)X^\eta\Big).
\end{align}
From  \eqref{varepsilon}, \eqref{Y}, \eqref{Delta}, \eqref{M}, \eqref{P},
\eqref{Idef}, \eqref{I1est3}, Lemma \ref{Lowerbound}
and  choosing   $\eta<\delta$  we find
\begin{equation}\label{I1est4}
I_1\gg Y^3X^{3-c}.
\end{equation}
Finally \eqref{I1est1} and \eqref{I1est4} give us
\begin{equation}\label{Gammaest}
\Gamma(X)\gg\varepsilon Y^3X^{3-c}.
\end{equation}
Bearing in mind \eqref{varepsilon}, \eqref{Y} and \eqref{Gammaest}
we establish that $\Gamma(X)\rightarrow\infty$ as $X\rightarrow\infty$.

The proof of the Theorem \ref{MyTheorem} is complete.

\vskip20pt
\footnotesize
\begin{flushleft}
S. I. Dimitrov\\
Faculty of Applied Mathematics and Informatics\\
Technical University of Sofia \\
8, St.Kliment Ohridski Blvd. \\
1756 Sofia, BULGARIA\\
e-mail: sdimitrov@tu-sofia.bg\\
\end{flushleft}

\end{document}